\documentclass[11pt]{amsart}
\usepackage{amscd,amssymb,amsopn,amsmath,amsthm,graphics,amsfonts,enumerate,verbatim,calc}
\usepackage[dvips]{graphicx}
\usepackage[colorlinks=true,linkcolor=red,citecolor=blue]{hyperref}
\usepackage[usenames, dvipsnames]{color}

\newcommand{\ep}{\varepsilon}
\newcommand{\ra}{\rightarrow}
\newcommand{\xs}{x_{\ast}}
\newcommand{\vs}{v_{\ast}}
\newcommand{\norm}{\|}
\newcommand{\bd}{\textrm{bdry\;}}
\newcommand{\txm}{T_{x}M}
\newcommand{\expxy}{\exp_{x}^{-1}{y}}
\newcommand{\pro}{proj}
\newcommand{\proj}{\textrm{proj}_S}
\newcommand{\li}{\langle}
\newcommand{\ri}{\rangle}
\newcommand{\expx}{\exp_{x}}
\newcommand{\fii}{\varphi}
\newcommand{\RR}{\mathbb{R}}

\newcommand{\landa}{\lambda}

\newtheorem{theorem}{Theorem}[section]
\newtheorem{corollary}[theorem]{Corollary}
\newtheorem{lemma}[theorem]{Lemma}

\newtheorem{proposition}[theorem]{Proposition}

\theoremstyle{definition}
\newtheorem{definition}[theorem]{Definition}
\newtheorem{example}[theorem]{Example}

\newtheorem{remark}[theorem]{Remark}

\theoremstyle{plain}

\theoremstyle{definition}

\numberwithin{equation}{section}

\begin{document}
\title[Convexity of the distance function]{ Convexity of the distance function to convex subsets of Riemannian manifolds}
\author[S. Khajehpour]{S. Khajehpour}
\address{ School of Mathematics,
 Institute for Research in Fundamental Sciences (IPM),
P. O. Box  19395-5746, Tehran, Iran}
\email{solmazkh114@ipm.ir}
\author[M. R. Pouryayevali]{M. R. Pouryayevali}
\address{ Department of Mathematics,
 University of Isfahan,
P. O. Box  81745-163, Isfahan, Iran}
\email{pourya@math.ui.ac.ir}


\thanks{This research was supported by a grant from IPM}
\maketitle
\begin{abstract}
A characterization of proximal normal cone is obtained and a separation theorem for convex subsets of Riemannian manifolds is established. Moreover, the convexity of the distance function $d_S$ for convex subset $S$ in the cases where the boundary of $S$ contains a geodesic segment, the boundary of $S$ is $C^2$ or the boundary of $S$ is not regular is discussed. Furthermore, a nonsmooth version of positive semi-definiteness of Hessian of convex functions on Riemannian manifolds is established.
\end{abstract}

\section{Introduction}\label{int}
Convexity is an old and important notion of mathematics which has a key role in many area of geometry such as properties of projection map\cite{Federer, Walter}, geometrical and topological restriction of Riemannian manifolds\cite{Green-Sh, Yau} and Monge-Amper equations\cite{mongeampere, Guan}.
Since the development of the theory of nonsmooth analysis, the notion of convexity have been studied widely from different point of view and its various applications have been employed in many branches of mathematics such as weak solutions of partial differential equations, variational analysis and optimization; see\cite{Rockafellar}.

This paper addresses the question of how the curvature affects the convexity of distance function. Indeed, it is well known  that if $M$ is a Hadamard manifold and $S\subseteq M$ is a closed convex subset, then the distance function $d_S$ is convex; see\cite{sakai}. Hence it is natural to ask if $M$ is an arbitrary Riemannian manifold (without any assumption on the curvature of $M$) the distance function $d_S$ is convex? In fact, understanding how curvature affect the convexity of a given function is a basic question. However, by considering  a  short segment of a great circle  on sphere as a convex set $S$, one can see that $d_S$  is not convex  on any open neighborhood of $S$. Similar examples can be  found  on every manifolds of  positive curvature.  In this paper we  explore why convexity of distance function $d_S$ fails in these examples and by imposing some restriction on $S$, the convexity of distance function $d_S$ is discussed. For this aim we first study the behavior of  closed convex subsets of Riemannian manifold $M$ and to deal with the boundary of $S$ we use the tools from nonsmooth analysis such as proximal normal cone. Indeed, we show that proximal normal cone is a generalization of normal bundle of a Riemannian submanifolds of $M$ for arbitrary subsets of $M$. By the use of this fact, we characterize the members of  proximal normal cone by projection map and derive
the nonsmooth analogous of tubular neighborhood theorem in Riemannian manifolds.
We refer the reader to \cite{Clark book}, where the concept of proximal normal cone was considered in Hilbert spaces. Studying  this problem leads us to obtain a separation theorem for convex subsets of Riemannian manifolds which help us to study  the convexity of the distance function by means of support principle for convexity. Moreover, the well known result regarding the positive definiteness of Hessian of a $C^2$ convex function on Riemannian manifolds is generalized to continuous convex functions by using the second order superjets. The paper is organized as follows. Section 2 is concerned with the proximal normal cone and metric projection in Riemannian manifolds. A characterization of proximal normal cone provide a separation theorem in the setting of manifolds. In Section 3, by imposing certain condition on the boundary of locally convex subsets of Riemannian manifolds, we discuss the convexity of the distance function to these subsets.

 Let us give a quick review of notations and concepts we need in the sequel. In this paper, we use the standard notation and known results of Riemannian
manifolds; see, e.g., \cite{docarmo, sakai}. Throughout this paper, $(M, \li.,.\ri)$ is a complete finite dimensional Riemannian manifold and $d(x, y)$ is the Riemannian distance on $M$. We denote the interior, closure and boundary of a subset $A$ of $M$ by $A^{\circ}$, $\overline{A}$ and $\bd A$, respectively.

A subset $S$ of $M$ is said to be convex if for every $x,y\in S$ there exists a unique minimizing geodesic from $x$ to $y$ lying entirely in $S$. By the  Whitehead theorem  there exists a  convex neighborhood around each point in $M$. In such a neighborhood we have parallel transport $l_{xy}:T_{x}M\ra T_{y}M$, that is, a linear isometry which sends each vector $v_{x}$ to its unique parallel vector $w_{y}$.  It is easy to see that $l_{xy} \expxy= -\exp_{y}^{-1}x$. Besides one can see that the covariant differentiation of a vector field $V(t)$ along a short geodesic $\gamma(t)$ can be computed by parallel transport, by the following formula
\begin{equation}\label{p transport}
V^{\prime}(t_0)=\lim_{t\ra t_0}\frac{l_{\gamma(t)\gamma(t_0)}V(t)-V(t_0)}{t-t_0}.
\end{equation}
The convexity radius of $M$ at $x$ is denoted by $r(x)$. Indeed,
\begin{align*}
r(x)=\sup \{& r>0: \textrm{ any metric ball in $B(x,r)$ is convex and any geodesic}\\
 &\textrm{ segment in $B(x,r)$ is a minimal geodesic joing its end points }\}.
\end{align*}
 A closed subset $S$ is called locally convex if for every $x\in S$, there exist $0<\ep(x)\leq r(x)$ such that $S\cap B(x,\ep(x))$ is convex. Note that by Cartan-Hadamrd theorem, the notions of convexity and local convexity agree in a Hadamard manifold. A real valued function $f$ on an open set  $U\subset M$   is called convex  if $f \circ \alpha$ is convex for every geodesic $\alpha$ in $U$. One can see that a $C^2$ function $f$ is convex on $U$ iff $d^2f(x)\geq 0$  for every $x\in U$, where $d^2f$ is the Hessian of $f$ on Riemannian manifold $M$ which is defined by
 $$d^2f(v,w)=\li \nabla_{v}\nabla f, w\ri,$$
 for every $v,w \in \txm$.

We also use the notion of Jacobi fields. A vector field along a geodesic satisfying Jacobi equation is called a Jacobi field. If $\alpha:[0, a]\times [-\ep, \ep]\ra M$ is a variation, the length and energy functionals are defined by
$$L(s):=L(\alpha_{s})= \int_{0}^{a} \norm \dot{\alpha}_{s}(t)\norm dt,$$
and
$$E(s):=E(\alpha_{s})= \int_{0}^{a} \norm \dot{\alpha}_{s}(t)\norm^{2} dt.$$
Notice that if $\alpha$ is a variation through geodesic then
\begin{equation}\label{enery-lenghth}
 L(s)^{2}= aE(s).
\end{equation}
The following lemma which its proof is similar to \cite[Lemma 12.3.1]{docarmo} with slight modification, is substantial in the next section.
\begin{lemma}\label{n triangle ineq}
Suppose that $M$ is a complete Riemannian manifold such that all its sectional curvature is bounded above by a constant $\delta>0$. Suppose that $U$ is a convex subset of $M$ and $a$, $b,$ and $c$ are three distinct points in $U$ which do not lie on a geodesic segment. Let $d(a,c)<\frac{\pi}{\sqrt{\delta}}$ and $d(b,c)<\frac{\pi}{\sqrt{\delta}}$. These three points determine a unique geodesic triangle $\triangle(a,b,c)$ in $U$ with vertices $a$, $b$ and $c$. Let $\alpha$, $\beta$ and $\gamma$ be the angles of vertices $a$ ,$b$ and $c$, respectively. Let $A$, $B$ and $C$ be the lengths of the sides opposite the vertices $a$, $b$ and $c$, respectively. Then
$$\cos(\sqrt{\delta}A)\cos(\sqrt{\delta}B)+ \sin(\sqrt{\delta}A)\sin(\sqrt{\delta}B)\cos(\gamma)\geq \cos(\sqrt{\delta}C).$$
\end{lemma}
\section{Proximal normal cone and metric projection}
Let us begin this section with the definition of the proximal normal cone of subsets of
Riemannian manifolds; see \cite{Hosseini}.
\begin{definition}\label{n proximal cone}
Let $S$ be a closed subset of Riemannian manifold $M$ and $x\in S$. We define the proximal normal cone of $S$ at $x$ by
\begin{align*}
N^{P}_{S}(x) =\{v\in T_xM: \li v, \expxy \ri \leq\sigma\|\expxy \|^{2} ~ \text {for ~some}& \;\sigma\geq0\\ &~\text{and }\;y\in U\cap S\},
\end{align*}
where $U$ is a normal neighborhood of $x$ in $M$.
\end{definition}
 In fact $N^{P}_{S}(x)= N^{P}_{\expx^{-1}(S\cap U)}(0)$ where $N^{P}_{\expx^{-1}(S\cap U)}(0)$ is defined in \cite{Clark book}.
 Let us recall here the metric projection of a point $x\in M$ to a closed subset $S\subseteq M$ as
$$\proj x=\{y\in S: d(y,x)=d_{S}(x)\},$$
where $d_{S}(x)= \inf\{ d(y,x): y\in S\}$.\\
Since in every complete finite dimensional Riemannian manifold the closed balls are compact, we have the following lemma.
\begin{lemma}\label{n projection property}
For every $x\in M\setminus S$, $\proj x\neq \emptyset$ and $\{s\in \proj x: x\in M\setminus S\}$ is dense in $\bd S$.
\end{lemma}
The following theorem states the relationship between proximal normal cone and metric projection map; see also \cite{PR}.
\begin{theorem}\label{n proximal eq}
For every $x\in \bd S$,
\begin{align*}
N_S^P(x)=\{\dot{\gamma}(0)\bigl |  \gamma:[0, \ep]\ra M \textrm{ s.t. } \gamma(0)=x,  & ~ x\in \text{\pro}_S  \gamma(\ep)\\
&\textrm{ and }\gamma \textrm{ is  a minimizing geodesic}\}.
\end{align*}
\end{theorem}
\begin{proof}
Let $z\in M\setminus S$ and $x\in \proj z$. Suppose that $\gamma:[0,\ep]\ra M$ is a minimizing geodesic joining $x$ to $z$. Let $U=B(x,r)$ be a convex neighborhood of $x$ and $0<t_{0}\leq \ep$ be such that $y_0:=\gamma(t_0)\in B(x,r)$. For $y\in U$ define  $\fii(y)=d^2(y,y_0)$.  Since $\gamma$ is a minimizing geodesic, it is obvious that $x\in \proj y_0$ and we get $\fii(x)\leq \fii(y)$ for all $y\in S\cap U$. By considering Taylor's theorem, for all $y\in U$ we have
$$\fii(y)=\fii(x)+\li d\fii(x), \expxy\ri + o(\norm \expxy\norm).$$

Note that $d\fii(x)=-2\expx^{-1} y_0$ (see\cite[p.108]{sakai}). Since $\fii$ is smooth on $U$ there exits $M\geq0$ such that for every $y\in U$ we have
$$o(\norm \expxy\norm) \leq M \norm \expxy \norm ^2.$$
which implies
\begin{equation}
\fii(y)-\fii(x)\leq -2\li \expx^{-1} y_0, \expxy \ri+ M \norm \expxy\norm ^2,\qquad \forall y\in S\cap U
\end{equation}
This confirms the proximal normal inequality for $\expx^{-1} y_0$. Since $N_S^P(x)$ is a cone, we conclude that $\dot{\gamma}(0)\in N_S^P(x)$.

To prove the converse inclusion, suppose that $v\in N_S^P(x)$. By Definition \ref {n proximal cone} there exists $\sigma \geq 0$ such that for a convex neighborhood $U=B(x,R)$  we have
\begin{equation}\label{n 1}
\li v, \expxy\ri \leq \sigma \norm \expxy\norm^2\qquad \forall y\in S\cap U.
\end{equation}
Take $\gamma(t):= \expx tv$ and $y_0:= \expx \ep v \in U$. Notice that
$$d_S(y_0)\leq d(y_0, x)=\ep \norm v\norm =:r.$$
Let $\ep$  be small enough such that $B(y_0, 2r)\subseteq U$. We claim that $x\in \proj y_0$. Define the function  $\fii(y):= d^2(y,y_0)$ for every $y\in U$. By using Taylor's theorem, we have
$$\fii(y)=\fii(x)+ d\fii(x)(\expxy)+\frac{1}{2}d^2\fii(x)(\expxy)^2+o(\norm \expxy\norm^2).$$
Since $d\fii(x)= -2\expx^{-1} y_0$ we get
$$d\fii(x)(\expxy)= -2\ep\li v,\expxy\ri$$
Put $w:=\expxy$ and $\theta(s)=\expx sw$ and consider the variation $\alpha$ through geodesics by $\alpha_s(0)=y_0$ and $\alpha_s(r)= \theta(s)$. Therefore $L(s)= d(\theta(s), y_0)$ and according to (\ref{enery-lenghth}) we have
$$ d^2\fii(x)(w)^2=\frac{d^{2}}{ds^{2}}\biggl|_{s=0}\fii(\theta(s))= \frac{d^{2}}{ds^{2}}\biggl|_{s=0}L^2(s)=rE^{\prime\prime}(0).$$
Now we apply second variation formula for the variation $\alpha$ and its variation field $V(t):=\frac{\partial\alpha}{\partial s}(0,t)$ and get
$$\frac{1}{2}E^{\prime\prime}(0)= -\int_{0}^{r}\li V,V^{\prime\prime}+R(\dot{\alpha_0},V)V)\ri dt + \li V(t), V^{\prime}(t)\ri \bigl|_{t=0}^{t=r}+ \li(\frac{\partial \alpha}{\partial s})^{\prime}(0,t), \dot{\alpha_0}(t)\ri\bigl|_{t=0}^{t=r}.$$
Note that $V$ is a jacobi field, because $\alpha$ is a variation through geodesics and thus
$$V^{\prime\prime}+R(\dot{\alpha_0},V)V=0.$$
Moreover,
$$(\frac{\partial \alpha}{\partial s})^{\prime}(0,0)=(\frac{\partial \alpha}{\partial s})^{\prime}(0,r)=0.$$
It remains to compute $V^{\prime}(r)$. Assume that $\sum w_i\partial_i$ is the coordinate representation of $w$ in a normal coordinate around $y_0$. Sine $V$ is the unique Jacobi field along $\alpha_0$ with $V(0)=0$ and $V(r)= w$, we can write it in this normal coordinate by
$$V(t)=\sum \frac{t}{r}w_i\partial_i\bigl|_{\alpha_0(t)}.$$
Hence
$$V^{\prime}(t)= \frac{1}{r}\sum w_i\partial_i\bigl|_{\alpha_0(t)}+ \chi(t),$$
where
$$\chi(t)=\frac{t}{r}\big(\sum w_i\partial_i\big)^{\prime}.$$
This leads to
$$d^2\fii(x)(w)^2= rE^{\prime\prime}(0)=r\li V(r),V^{\prime}(r)\ri= \norm w\norm^2+r\li w,\chi(r)\ri$$
and then
\begin{equation}\label{1}
\fii(y)=\fii(x)-2\ep\li v, \expxy\ri+ \norm \expxy\norm^2+ r\li\expxy, \chi(r)\ri+o(\norm \expxy\norm^2).
\end{equation}
By taking $r$ small enough,  we can ensure that (since $\chi$ is a smooth vector field along $\alpha_0$ and $\chi(0)=0$) there exists $-\frac{1}{2}<\delta< \frac{1}{2}$ such that
\begin{equation}\label{2}
r\li \expxy,\chi(r)\ri=\delta \norm \expxy\norm^2.
\end{equation}
Note that $\fii$ is smooth on $U$ and therefore there exists $M>0$ such that for all $y\in U$
\begin{equation}\label{3}
|o(\|\expxy\|^{2})| < M\|\expxy \|^{3}.
\end{equation}
By shrinking $U=B(x,R) $ if necessary one can assume that $1+\delta-MR>0$ which implies $2\sigma\ep\leq 1+\delta-MR$ for small enough $\ep$  and then
$$2\sigma\ep\leq 1+\delta-M\norm \expxy\norm.$$
It follows from (\ref{1})-(\ref{3}) that
\begin{align*}
2\sigma\ep \norm \expxy\norm^2 &\leq \norm \expxy\norm^2+\delta\norm \expxy\norm^2-M\norm \expxy\norm^3\\
                               &\leq \norm \expxy\norm^2+r\li \expxy,\chi(r)\ri+ o(\|\expxy\|^{2})\\
                               &=\fii(y)-\fii(x)+2\ep\li v,\expxy\ri.
\end{align*}
According to (\ref{n 1}), one can deduce that for $y\in S\cap U$ we have
\begin{align*}
2\sigma\ep \norm \expxy\norm^2 &\leq \fii(y)-\fii(x)+2\sigma\ep\norm \expxy\norm^2
\end{align*}
Thus $d(x,y_0)\leq d(y, y_0)$ for $y\in S\cap U$ and consequently $x\in \textrm{proj}_{S\cap U}y_0$. Since $d(x,y_0)=r$ and $B(y_0, 2r)\subseteq U$ we conclude that $x\in \proj y_0$.
\end{proof}
 In  the following proposition we consider the codimension-1 submanifolds of Riemannian manifold $M$.

\begin{proposition}\label{n level set}
Let $f:M\ra \RR$ be a $C^1$ function and let $S:= f^{-1}(c)$ for some $c\in \RR$. If $\nabla f(x)\neq 0$ for every $x\in S$, then $N^P_S(x)\subseteq \{\nabla f(x)\}$.
\end{proposition}
\begin{proof} Since $N^{P}_{S}(x)= N^{P}_{\expx^{-1}(S\cap U)}(0)$ and  $d \expx(0)=Id$, according to \cite [Proposition 1.9]{Clark book} the proof is obvious.
\end{proof}
Note that Proposition \ref{n level set} does not guarantee $N_S^P(x)\neq \{0\}$, although $N^P_S(x)=\{\nabla f(x)\}$ if $f$ is a $C^2$ function.  Since every $C^2$ submanifold is locally a level set of a $C^2$ function, we have the following corollary.
\begin{corollary}\label{level set}
Let $W$ be a $C^2$ codimension-1 submanifold of a Riemannian manifold $M$ and $N_xW$ be the normal bundle of $W$ at $x\in W$ . Then $N^P_W(x)=N_xW$ for every $x\in W$.
\end{corollary}
For more details regarding proximal normal cone of $C^{1,1}$ subsets of Riemannian manifolds see \cite{pde}.

We know that each point in the boundary of a  convex set $S\subseteq \RR^n$ lies in some hyperplane such that $S$ is contained in one of the associated hyperspaces. The  following propositions show that we have the same result for convex subsets of Riemannian  manifold $M$ in tangent vector spaces $\txm$ of every $x\in \bd S$.
\begin{proposition}\label{n convex subset}
Let $S$ be  locally convex, then $v\in N^P_S(x)$ if and only if
$$  \li v, \expxy\ri\leq 0\qquad \forall y\in S\cap B(x,\ep(x)).$$
\end{proposition}
\begin{proof}
  The ``if'' part follows from Definition \ref{n proximal cone}. To prove the converse statement, let $v\in N^P_S(x)$. This means that there exists $\sigma \geq 0$ such that for all $y\in S\cap B(x,\ep(x))$ we have
$$\li v, \expxy\ri\leq \sigma \norm \expxy \norm ^2.$$
The convexity of $S\cap B(x,\ep(x))$ implies $\expx t\expxy \in S\cap B(x,\ep(x))$ for $y\in S\cap B(x,\ep(x))$ and $t\in [0,1]$. Thus
$$\li v, \expxy\ri\leq \sigma t \norm \expxy\norm ^2.$$
By letting $t\ra 0$ the proof is complete.
\end{proof}
\begin{corollary}\label{c1}
Let $S$ be a closed  convex subset of  Riemannian manifold $M$. Then
$$v\in N^P_S(x) \text{ if and only if } \li v, \expxy\ri\leq 0, \qquad \textrm{for every }y\in S.$$
\end{corollary}
\begin{proposition}\label{n non empty}
Let $S$ be a closed locally convex subset of  complete Riemannian manifold $M$. Then for every $x\in \bd S$, $N^P_S(x)\neq\{0\}$.
\end{proposition}
\begin{proof}
Let $x\in \bd S$. Since $N^P_S(x)=N^P_{S\cap \bar{U}}(x)$ for every closed convex neighborhood $\bar{U}$ of $x$, we can assume that $S$ is strongly convex.  By Lemma \ref{n projection property}  and Theorem \ref{n proximal eq} there exists a sequence $\{x_i\}\subseteq \bd S \cap U$ such that $x_i \ra x$ and $N_S^P(x_i)\neq \{0\}$, where $U$ is a convex neighborhood of $x$. Suppose that $\xi_i\in N^P_S(x_i)$ with $\norm\xi_i\norm=1$. Corollary \ref{c1} yields
$$\li \xi_i, \exp_{x_i}^{-1}y\ri \leq 0\qquad \textrm{for every }y\in S.$$
Put $\xi_i^\prime:=l_{x_{i}x}\xi_{i}$ and since $\{\xi_i^\prime\}$ is a bounded sequence in $\txm$, we extract a subsequence converging to $\xi_0$, without relabeling.  Thus
$$\li \xi_i^\prime ,l_{x_{i}x}\exp_{x_i}^{-1}y\ri \leq 0\qquad \textrm{for every }y\in S.$$
Since the maps $z\mapsto l_{zx}$ and $z\mapsto \exp_z^{-1}y$ for $z\in U$  where $y$  is fixed point in $S$, are continuous we get
$$\li \xi_0, \exp_x^{-1}y\ri\leq 0 \qquad \textrm{for every }y\in S.$$
Hence by Corollary \ref{c1} the proof is complete.
\end{proof}
 In \cite{Walter} it was proved that around any convex subset $S$ of $M$, there exists a tubular neighborhood $W$ of $S$ which the projection map is single- valued on $W$. In the following proposition we characterize the elements of  $W$ by means of the curvature of $M$ around every boundary point of $S$.
\begin{proposition}\label{projection}
Let $S$ be a closed locally convex subset of Riemannian manifold $M$. Let all  sectional curvature  of $M$ on $\overline{B(x,\ep(x))}$ is bounded above by a constant $\delta_x>0$. Put $t_x:=\min\{\frac{\pi}{2\sqrt{\delta_x}}, \frac{\ep(x)}{2}\}$ and
$$W:= \{\expx tv \bigl| x\in \bd S, v\in N^P_S(x)\textrm{ with } \norm v \norm=1, t<t_x\}.$$
Then the map $\proj$ is single-valued map on $W$. Moreover, for every $x\in \bd S$ and $v\in N^P_S(x)$ with $\norm v\norm=1$ there exists $t\geq 0$ such that $\expx tv\in W^{\circ}$.
\end{proposition}
\begin{proof}
Let $y=\expx tv\in W$ and $x_1,x_2\in \proj y$. It is clear from the definition of $W$ that $x_1,x_2\in B(x,\ep(x))$. Now consider the triangle $\triangle(x_1,x_2,y)$ in convex neighborhood $B(x,\ep(x))$ with angel $\alpha$ associated with vertex $x_1$. Let $A:=d(x_1,y)=d(x_2,y)$ and $B:= d(x_1,x_2)$. Obviously $A< \frac{\pi}{2\sqrt{\delta}}$ and $B\leq 2A<\frac{\pi}{\sqrt{\delta}}$. According to Lemma \ref{n triangle ineq}
$$\cos(\sqrt{\delta}A)\cos(\sqrt{\delta}B)+ \sin(\sqrt{\delta}A)\sin(\sqrt{\delta}B)\cos\alpha \geq \cos(\sqrt{\delta}A).$$
Using Theorem \ref{n proximal eq} one has $\exp_{x_1}^{-1}y\in N^P_S(x_1)$ and thus by Proposition \ref{n convex subset} $\cos\alpha \leq0$ which implies
$$\cos(\sqrt{\delta}A)\cos(\sqrt{\delta}B)\geq \cos(\sqrt{\delta}A).$$
Hence $\cos(\sqrt{\delta}B)=1$ and then $B=0$. Therefore $\proj$ is single valued on $W$.

Suppose that for every $x\in S$, $\ep(x)\leq r(x)$ is the largest value such that $B(x,\ep(x))\cap S$ is convex and suppose that $K\subseteq \bd S$ is compact. One can see for every sequence $\{x_i\}\subset K$ with $x_i\rightarrow x_0$, if $d(x_i,x_0)<\frac{\ep(x_0)}{2}$ and $r(x_i)>\frac{\ep(x_0)}{2}$ then $\ep(x_i)\geq \frac{\ep(x_0)}{2}$ which implies $\inf _{x\in K}\ep(x)>0$. Thus $\inf_{x\in K}t_x>0$( Note that $\sup_{x\in K} \delta_x<+\infty$ ).
Now let $x\in \bd S$  and $v\in N^P_S(x)$ with $\norm v\norm =1$. Put $m:= \inf _{x\in \overline{B(x,l)}\cap \bd S}t_x>0$. Take $t>0$ and $\delta>0$ such that $t+ \delta<m$ and $2\delta+t<\frac{l}{2}$. Put $y:=\expx tv$ and suppose that $z\in B(y,\delta)$. Let $x_1\in \proj z$. Note that
$$d(z,x_1)\leq d(z,x)\leq d(z,y)+d(y,x)<\delta+t$$
and
$$d(x,x_1)\leq d(x,y)+d(y,x_1)\leq 2d(y,x_1)\leq 2( d(y,z)+d(z, x_1))< 2(2\delta+t).$$
Our choices of $\delta$ and $t$ make $x_1\in B(x,l)$ and therefore $d(z,x_1)< t_{x_1}$. Thus $z\in W$ and $y\in W^{\circ}$.
\end{proof}
\begin{remark}\label{projection2}
By a a similar observation as in the proof of Corollary \ref{projection} one can show that for every $y:=\expx tv\in W$ we have $\proj y=\{x\}$.
\end{remark}
\begin{corollary}\label{coro}
Let $S$ be a closed locally convex subset of Riemannian manifold $M$. Then there exists an open set $W$ containing $S$ such that $\proj$ is single-valued continuous mapping on $W$.
\end{corollary}

\begin{definition}\label{support hyperplane}
Let $S$ be a closed subset of complete Riemannian manifold $M$. Let $\xs\in \bd S$ and $v\in N^P_S(\xs)$ with $\norm v\norm =1$. We say $S$ is supported at $\xs$ relative to $v$ when there exist a smooth codimension-1 submanifold $H$ and $r>0$ such that $\xs\in H$ and for every $x=\exp_{\xs}tv$ with $\xs=\proj x$ and $t<r$ we have
\begin{itemize}
  \item[(1)] $d_H(x)=d_S(x);$
  \item[(2)] $d_H(y)< d_S(y)$
\end{itemize}
  for every $y\in U_x$ where $U_x$ is a small enough neighborhood of $x$.
\end{definition}
 Although Propositions \ref{n convex subset} and \ref{n non empty} are not about separation of convex subsets on ambient manifold, they imply the following theorem.
\begin{theorem}\label{n seperation}
Let $S$ be a closed locally convex subset of complete  Riemannian manifold $M$. Suppose that all sectional curvature of $M$ is bounded above by a constant $\delta>0$. Then $S$ is supported at every $\xs\in \bd S$ relative to $v\in N^P_S(\xs)$ with $\norm v\norm =1$.
\end{theorem}
\begin{proof}
Let $\xs\in \bd S$ and $v\in N^P_S(\xs)$ with $\norm v\norm=1$. By proposition \ref{n convex subset} we have
$$\li v, \exp_{\xs}^{-1}y \ri \leq 0, \qquad \textrm{for every }y \in S\cap B(\xs,\ep(\xs)).$$
Put $r=\min\{\frac{\ep(\xs)}{2}, \frac{\pi}{2\sqrt{\delta}}\}$, and let $H_0$ be the hyperplane in $T_{\xs}M$ with the normal vector $v$, i.e. $H_0=\{w\in T_{\xs}M: \li v,w\ri=0\}$. Take
\begin{equation}\label{H}
H:= \exp_{\xs}H_0\cap \overline{B(\xs, 2r)}.
\end{equation}
Note that $\xs\in H$, $H$ is a smooth codimension-1 submanifold of $M$ and $v\in N_{\xs}H$. Suppose that $W$ is the tubular neighborhood around $S$ mentioned in Proposition \ref{projection} and let $x= \exp_{\xs}tv\in B(\xs,r)\cap W$ such that $\xs=\proj x$. Take
$$U_x:=\{y\in B(\xs, r)\cap W:\li \exp_{\xs}^{-1}y, v\ri >0\},$$
and let $y\in U_x$ and $\gamma_1(t)$ be the unique minimizing geodesic joining $y$ and $y_{\ast}$. If we consider $\tilde{\gamma}_1(t):= \exp_{\xs}^{-1}\gamma_1(t)$, then our choice of $U_x$ grantees that $\tilde{\gamma}_1$ intersects $H_0\cap B(0_{\xs}, 2r)$. Thus $d_H(y)< d_S(y)$ for every $y\in U_x$.\\
It remains to show that $d_H(x)=d_S(x)$. To do so, note that  by Corollary \ref{level set}, $N_H^P(y)=\{\landa n(y): \landa \in \RR\}$ where $n(y)$ is a unite vector in normal vector bundle $N_yH$. We choose $n(y)$ such that $n(\xs)=v$ and $y\mapsto n(y)$ is a continuous vector field on a neighborhood of $H$. Suppose that $d_H(x)<d_S(x)$. According to Theorem \ref{n proximal eq} there exists $z\in H$ such that $t^{\prime}:=d_H(x)=d(x,z)<t$ and $x=\exp_{z}t^{\prime}n(z)$. Note that  $x$ and $z$ are in a convex neighborhood of $\xs$. We consider the geodesic triangle $\triangle(x,\xs,z)$ in the convex neighborhood $B(\xs,\ep(\xs))$ with the angels $\alpha$, $\beta$ and $\theta$ at vertices $\xs$, $z$ and $x$, respectively.
Since $v\in N_{\xs}H$, $n(z)\in N_{z}H$ and the unique minimizing geodesic joining $\xs$ and $z$ lies entirely in $H$, it follows that $\alpha=\beta= \frac{\pi}{2}$.
 Our choices of $W$ and $H$ imply
 $$d(x,\xs)=t<\frac{\pi}{2\sqrt{\delta}},\qquad d(x,z)=t^{\prime}<\frac{\pi}{2\sqrt{\delta}},\qquad t^{\prime\prime}:=d(\xs,z)<\frac{\pi}{\sqrt{\delta}}.$$
By Lemma \ref{n triangle ineq} we have
$$\cos(\sqrt{\delta}l)\leq \cos(\sqrt{\delta}l^{\prime\prime})\cos(\sqrt{\delta}l^{\prime})$$
and
$$\cos(\sqrt{\delta}l^{\prime})\leq \cos(\sqrt{\delta}l^{\prime\prime})\cos(\sqrt{\delta}l)$$
Since $\sqrt{\delta}l<\frac{\pi}{2}$ and $\sqrt{\delta}l^{\prime}<\frac{\pi}{2}$, we have $\cos(\sqrt{\delta}l)+\cos(\sqrt{\delta}l^{\prime})>0$. Thus $\cos(\sqrt{\delta}l^{\prime\prime})\geq 1$ and we conclude that $l^{\prime\prime}=0$ and $z=\xs$ which complete the proof.
\end{proof}
\begin{remark} The proof of Theorem  \ref{n seperation} shows how curvature affects the radius of the convex neighborhood $B(x_{*}, 2r)$ used in (\ref{H}). In fact the greater curvature implies the less radius. Especially  if $M$ is a Hadamard manifold, then $W= M\setminus S$  in  Proposition \ref{projection} and  $H= \exp_x H_0$ in Theorem \ref{n seperation} for every $x\in \bd S$. This means that the separation theorem does hold on Hadamard manifolds like Euclidean spaces.
\end{remark}
\section{Convexity of distance function to a convex subset}
 In this section we study the convexity of distance function $d_S$ where $S$ is a locally convex subset of $M$. We will show that how curvature of $M$ affect this problem.
 \begin{theorem}\label{concave}
Let $M$ be a Riemannian manifold such that  its sectional curvature is positive. Suppose that $S$ is a convex subset which its boundary contains  a geodesic  segment of $M$. Then there exists a geodesic $\alpha:(-\ep,\ep)\ra M\setminus S$ such that $d_S \circ \alpha(t)$ is a concave function.
\end{theorem}
 \begin{proof}
According to Proposition \ref{projection} there exists a tubular neighborhood $W$ of $S$ such that the the map $\proj:W\ra S$ is single valued and continuous. Now suppose that $x\in W\setminus S$ and $\proj(x):=\xs=\gamma(0)$ where $\gamma:[-\ep,\ep]\ra M$ is the image of a minimal geodesic whose image is contained in $S$. Put
$$v:=l_{\xs x}\vs $$
where $\vs:= \dot{\gamma}(0)$. Let $\alpha:(-\ep,\ep)\ra M$ be a geodesic in $M$ such that $\alpha(0)=x$ and $\dot{\alpha}(0)=v$. For each $s\in (-\ep, \ep)$, let $\beta_{s}:[0,1]\ra M$ be the unique minimizing geodesic joining the points $\gamma(s)$ and $\alpha(s)$. By letting $x$ close enough to $\xs$ take $\beta_s(t)= \exp_{\gamma(s)}t\exp_{\gamma(s)}^{-1}\alpha(s)$. Let us define the variation $\beta(s,t)= \beta_{s}(t)$ and $\beta(t):=\beta_{0}(t)$. Note that
$$l:=\norm \dot{\beta}(t)\norm= \norm \dot{\beta}(0)\norm= \norm \exp_{\xs}^{-1}x\norm= d_{S}(x)$$
If we denote the length of geodesic $\beta(s)$ by $L(s)$, then by the first variation formula we have
\begin{equation}\label{neq1}
\frac{d}{ds}\biggl|_{s=0}L(s)= \frac{1}{l}\li V(t),\dot{\beta}(t)\ri\biggl |_{t=0}^{t=1}-\frac{1}{l}\int_{0}^{1}\li V(t),(\dot{\beta}(t))^{\prime}\ri dt,
\end{equation}
where $V(t)= \frac{\partial \beta}{\partial s}(0,t)$ is the variation field of $\beta_{s}$.
Since $\beta$ is a geodesic, the last term in (\ref{neq1}) vanishes. Thus
$$\frac{d}{ds}\biggl|_{s=0}L(s)= \frac{1}{l}\bigg[ \li v, l_{\xs x} \exp_{\xs}^{-1}x\ri- \li \vs,\exp_{\xs}^{-1}x\ri \bigg]=0.$$
This shows that $s\mapsto L(s)$ has an extremum at $s=0$.

To compute the second derivative of $L(s)$, we use the second variation formula and (\ref{enery-lenghth}) to get
\begin{align*}
\frac{1}{2}\frac{d^{2}}{ds^{2}}\biggl|_{s=0}L^2(s)&=\frac{1}{2}E^{\prime\prime}(0)\\
&= \int_{0}^{1}(\li V^{\prime},V^{\prime}\ri- \li R(V,\dot{\beta})V,\dot{\beta}\ri) dt+ \li(\frac{\partial \beta_s}{\partial s})^{\prime}(0,t), \dot{\beta}(t)\ri\bigl|_{t=0}^{t=1}.
\end{align*}
Clearly the second term in the above statement vanishes. The first term is called the index form of $V$ and denoted by $I(V,V)$. Since $V$ is a jacobi field along $\beta$ then
$$I(V,V)\leq I(W,W)$$
where $W$ is a vector field along $\beta$ with $W(0)=V(0)$ and $W(1)=V(1)$ (see \cite{sakai}).
Consider $W$ as a parallel vector field along $\beta$ with $W(0)=\vs$. Obviously $W(1)=v$ and then
$$I(V,V)\leq I(W,W)= -\int_{0}^{1} \li R(W,\dot{\beta})W,\dot{\beta}\ri dt.$$
If  the sectional curvature of $M$ is positive we conclude that
$$\frac{1}{2}\frac{d^{2}}{ds^{2}}\biggl|_{s=0}L^2(s)= I(V,V)<0,$$
which shows that $L(s)$ has a maximum at $s=0$. Then
$$L(0)= d_{S}(x)= d_{S}(\alpha(0))> L(s)= d(\gamma(s), \alpha(s))\geq d_{S}(\alpha(s)).$$
This means that $d_{S}(\alpha(s))$ has a strict maximum at $s=0$  and is concave near $s=0$.
\end{proof}
Let us recall here the definition of geodesic point. Let $S$ be a Riemannian submanifold of Riemannian manifold $M$. We call $x\in S$ a geodesic point when for every $v\in T_xS$ the $M-$geodesic $\gamma_v$  starting from $x$ with velocity $v$ lies entirely in $S$,(see \cite{docarmo}).
\begin{corollary}\label{co}
Let $M$ be a Riemannian manifold such that  its sectional curvature is positive. Suppose that $S$ is a convex submanifold which its boundary contains a geodesic point. Then there exists a geodesic $\alpha:(-\ep,\ep)\ra M\setminus S$ such that $d_S \circ \alpha(t)$ is a concave function.
\end{corollary}
The next theorem investigates convexity of the distance function $d_S$ on an open neighborhood of $S$ where $S$ is a locally convex subset with regular boundary.
\begin{theorem}\label{last}
Let $M$ be  a complete Riemannian manifold and $S\subseteq M$ be a closed  locally convex set. Suppose that $W=\bd S$ is a $C^2$  codimension-1 submanifold and $n(x)\in N_xW$ is a unite vector such that $-n(x)\in N^P_S(x)$ for every $x\in W$. Let $h_x$ be the scalar second fundamental form of $W$ at $x\in W$. If $h_x>0$ for every $x\in W$, then there exists an open neighborhood $U$ of $S$ such that the distance function $d_S$ is convex on $U$.
\end{theorem}

Before proving Theorem \ref{last} we show that the hypothesis of positive definiteness of the scalar second form of the boundary is in fact very natural for locally convex subset with regular boundary.
\begin{proposition}\label{positive difinetness}
Let $S$ be a locally convex subset of Riemannian manifold $M$ and  $W=\bd S$ be a $C^2$  codimension-1 submanifold. Suppose that $n(x)\in N_xW $ is the inward pointing normal for every $x\in W$. Then $h_x\geq0$ for every $x\in W$.
\end{proposition}
\begin{proof}
By definition we require that
$$\li \nabla_v n(x),v\ri\leq 0\qquad \textrm{for every }x\in W \textrm{ and }v\in T_xW.$$
Suppose that $x\in W$ and $v\in T_xW$. Let $\gamma:(-\ep,\ep)\ra M\cap B(x, \ep(x))$ be a curve in  $W$ such that $\gamma(0)= x$ and $\dot{\gamma}(0)=v$. By Proposition \ref{n convex subset} we have
\begin{equation}\label{n eq1}
\li n(x), \exp_x^{-1}\gamma(t)\ri\geq 0,
\end{equation}
and
\begin{equation}\label{n eq2}
\li n(\gamma(t)), \exp_{\gamma(t)}^{-1}x\ri\geq 0.
\end{equation}
Put $v_t= l_{\gamma(t)x}n(\gamma(t))$. Thus by (\ref{n eq1}) and (\ref{n eq2}) we have
$$\li v_t, \exp_x^{-1}\gamma(t)\ri\leq0\leq \li n(x), \exp_x^{-1}\gamma(t)\ri.$$
Hence
$$\li \frac{v_t-n(x)}{t}, \frac{\exp_x^{-1}\gamma(t)}{t}\ri\leq0\qquad t\sim 0.$$
Put $\tilde{\gamma}(t)=\expx^{-1} \gamma(t)$ which is a curve in $\txm$ passing through $0_x$ with velocity $v$.
By letting $t\ra0$ we conclude that
$$\lim_{t\rightarrow 0}\frac{\expx^{-1}\gamma(t)}{t}= \dot{\tilde{\gamma}}(0)=v,$$
and by (\ref{p transport}) we get
$$\li \nabla_v n(x), v\ri\leq 0.$$
\end{proof}
\begin{proof}[Proof of Theorem \ref{last}]
Assume that $\widetilde{W}$ is the tubular neighborhood of $W$ such that the projection map is single valued by Proposition \ref{projection}. Put
$$U:=\{y\in \widetilde{W}\mid y=\expx tn(x) \textrm{ for } t>0 \textrm{ and } x\in W \},$$
and
$$V:=\{y\in \widetilde{W}\mid y=\expx tn(x) \textrm{ for } t<0 \textrm{ and } x\in W \}.$$
Let us define
\begin{align}
\varphi &: \widetilde{W}\ra \RR \\
\varphi(y)&=\left\{
         \begin{array}{ll}
              -d_W(y), & \hbox{$y\in U\cup W$;} \\
              d_W(y), & \hbox{$y\in V$.}
              \end{array}
            \right.
\end{align}
In  fact for  $t$ sufficiently small, by Remark \ref{projection2} we have $\fii(\expx tn(x))= -t$ which implies $\nabla \fii(x) =-n(x)$. Now suppose that $v,w\in T_xW$. By definition of second order derivative we have
$$d^2\fii(x)(v,w)= \li \nabla_v\nabla \fii(x),w\ri= \li \nabla_v -n(x), w\ri= \li \textbf{S}_x v,w\ri= h_x(v,w)$$
where $\textbf{S}$  is the shape operator of $W$. By simple calculation one can see that for $v\in \txm$ and its tangential component $\alpha\in T_xW$, we have
$$d^2\fii(x)(v)^2= d^2\fii(x)(\alpha)^2= h_x(\alpha, \alpha).$$
Therefore, if the scalar second fundamental form $h_x$ of $W$ is positive definite for every $x\in W$, then $\fii$ is a convex function in a neighborhood of $W$. Since $\fii(y)= d_S(y)$ for $y\in V$, we can conclude that under this assumption $d_S$ is a convex function on a neighborhood of $S$.
\end{proof}
\begin{example}
Consider $M=\mathbb{S}^2$ the unite sphere in $\RR^3$. Let $S=\{(x,y,z): z\geq \frac{1}{2}\}\subset M$. In fact $S=\overline{B(N,\frac{\pi}{3})}$  where $N$ is the north pole. It is obvious that $S$ is convex with $C^{\infty}$ boundary. By choosing inward normal vector $n(p)$ on $\bd S$, $h_p>0$ for every $p\in \bd S$; since otherwise $\bd S$ is a geodesic segment which is  a contradiction. Hence by Theorem \ref{last} there exist an open neighborhood $U$ of $S$ such that the distance function $d_S$ is convex on $U$.
\end{example}
\begin{example}
Suppose that $M$ is the paraboloid of revolution in $\RR^3$. Let $S_1=\{(x,y,z)\in M: z\leq z_0\}$  where $z_0\in (0,\infty)$. Consider  the function $f:M\ra \RR$ with $f(x,y,z)= x^2+y^2$. If we consider local parametrization of $M$ by  $(s,\theta)\mapsto (s\cos \theta,s\sin \theta, s^2)$, then $f(s,\theta)=s^2$ and
$$d^2f(s,\theta)= \left(
                                                                                                                                             \begin{array}{cc}
                                                                                                                                               \frac{2}{1+4s^2} & 0 \\
                                                                                                                                               0 & \frac{2s^2}{1+4s^2} \\
                                                                                                                                             \end{array}
                                                                                                                                           \right).$$
This shows that $d^2 f>0$ and therefore $S_1$ is a locally convex subset of $M$ with $h_p>0$ for every $p \in \bd S_1$ with inward pointing normal direction.
Thus by Theorem \ref{last}, $d_{S_1}$ is a convex function in a neighborhood of $S_1$.
\end{example}
Here we propose two approaches for considering convexity of function $d_S$ where the boundary of $S$ is not $C^2$. The first one  is the support principle for convexity which we will explain it in the following theorem. Recall that a function $f$ is said to be supported by a function $g$ at a point $p$ if $f(p)=g(p)$ and $g(q)\leq f(q)$ for all $q$ in a neighborhood of $p$. For more details one can see \cite{ Green}.
\begin{theorem}
A continuous function $f$ on a Riemannian manifold $M$ is convex if it is supported at each point $p$ by a convex function on some neighborhood of $p$.
\end{theorem}
\begin{remark}
By making  appropriate changes in Theorem \ref{n seperation}, one can find a neighborhood $W$ of $S$ such that for every $x\in W$ there exist a neighborhood $U_x\subseteq W$ and a smooth codimension-1  submanifold  $H$ satisfying $d_H(x)=d_S(x)$ and $d_H(y)\leq d_S(y)$ for every $y\in U_x$, i.e. $d_S$ is supported by $d_H$ at $x$.  If we construct $H$ such that $h_{\xs} >0$ (by imposing inward pointing normal) or equivalently if $\exp_{\xs}^{-1}H$ is strictly convex at $0_{\xs}$ in the Euclidean sense (see \cite[p. 372]{Gromove}) where $\xs=\proj x$, then it follows from  Theorem \ref{last} and the support principle for convexity  that $d_{S}$ is a convex function on $W$. This means that in the proof of Theorem \ref{n seperation}, let  $v=(0,...,0,-1)$ (by using some rotation in normal coordinate system around $\xs$), then if one can find $a>0$ such that $H_0=\{(w_1,...,w_n)\in T_{\xs}M|w_n=a(\sum_{i=1}^{n-1} w_i^2)$, then by Theorem \ref{n seperation} $d_S$ is supported by convex function $d_H$ which makes $d_S$ to be a convex function in a neighborhood of $S$.
\end{remark}
The next approach is more general and can be applied to every function on Riemannian manifolds. At first let us define second superjet of a function $f$ at every point $x$ in the domain of $f$, denoted by $J^{2,+} f(x)$, as
$$\big\{\big(d\fii(x),d^{2}\fii(x)\big)\mid \fii \in C^{2}(U), f-\fii \text{ attains a local maximum at } x\big\}.$$
Where $U$ is an open subsets of $M$ including the domain of $f$. Let $\Omega$ be an open convex subset of $M$ and $f:\Omega\ra \RR$ be a convex function. Let $\fii$ be a $C^2$ function such that $f(x_0)=\fii(x_0)$ and $f-\fii$ attains its maximum at $x_0\in \Omega$. Let $\gamma:[0,1]\ra \Omega$ be a geodesic such that for some $t_0\in(0,1)$ we have $\gamma(t_0)=x_0$. The convexity of $f$ implies that $f(\gamma(t))<(1-t)f(\gamma(0))+tf(\gamma(1)):=g(t)$ for $t\in[0,1]$. If $f(\gamma(t_0))<g(t_0)$, then by continuity of $\fii$ we can deduce that $\fii(\gamma(t))<(1-t)\fii(\gamma(0))+t\fii(\gamma(1))$ for $t$ sufficiently close to $t_0$. If $f(\gamma(t_0))=g(t_0)$ then by use of the mean value theorem for the function $\fii(\gamma(t))-g(t)$ one can see that this function attains its minimum at $t_0$ which implies that $d^2\fii(x_0)(\dot{\gamma}(t_0),\dot{\gamma}(t_0))\geq 0$. These show that if $f$ is convex and $(\xi, X)\in J^{2,+}f(x)$, then $X\geq0$. The next theorem states that the converse statement is true on manifolds with positive sectional curvature.

In the proof of the next theorem we use a special coordinate system called Fermi coordinate system. Let us explain the definition of this coordinate. Suppose that $\gamma:[a,b]\ra M$ be a unit speed minimizing geodesic in $M$. Assume that $\{e_i\}_{i=0}^{n-1}$ be an orthonormal basis for $T_{\gamma(a)}M$ with $e_0:=\dot{\gamma}(a)$. Let $e_i(t)$ be the parallel transport of $e_i$ to $\gamma(t)$. Note that $\{e_i(t)\}$ is an orthonormal frame for $T_{\gamma(t)}M$ for every $t\in [a,b]$. Consider the map $\Gamma: [a,b]\times \RR^{n-1}\ra M$ given by
$$\Gamma(t,x_1,...,x_{n-1})=\exp_{\gamma(t)}(\sum x_i e_i(t)).$$
One can see  $d\Gamma (t,0,...,0)$ is invertible for very  $t\in [a,b]$ which implies by inverse function theorem that there exist $\mu>0$ and $T_\mu\subseteq M$ such that $\Gamma:[a,b]\times B(0,\mu)\ra T_\mu$ is a diffeomorphism. Thus $(T_\mu,\phi)$ where $\phi:=\Gamma^{-1}$ is a coordinate system around $\gamma$ which is called Fermi coordinate system. For more details see \cite{Aubin, Gray}. Note that by similar observation as in the proof of Theorem \ref{n seperation}, one can see that $\proj \exp_{\gamma(t)}(\sum x_i e_i(t))=  \gamma(t)$ for every $t$ and $\norm (x_1,...,x_{n-1})\norm <\mu$ where $\mu>0$ is small enough and $S:=\gamma[a,b]$.

According to what we said before, the following theorem is a generalization of positive semi-definiteness of second order differential of a smooth convex function.
\begin{theorem}\label{jets}
Let $M$ be a Riemannian manifold with positive sectional curvature and  $\Omega$ is a convex subset of $M$. Let $f:\Omega\rightarrow \RR$ is continuous. Suppose that $X\geq0$  for every $(\xi, X)\in J^{2,+}f(x)$ and for every $x\in \Omega$. Then $f$ is a convex function.
\end{theorem}
\begin{proof}
Suppose that there exist a geodesic $\gamma:[0,1]\rightarrow \Omega$ such that $f\circ\gamma$ is not convex. Let $\lambda_0\in (0,1)$ is such that
$$f\circ \gamma(\lambda_0)>(1-\lambda_0)f(x_0)+\lambda_0f(y_0)$$
where $x_0:=\gamma(0)$ and $y_0:=\gamma(1)$. By shrinking the geodesic segment $\gamma[0,1]$ if necessary, one can assume that there exist $k,k_0\in \RR$ such that for every $\lambda\in [0,1]$ we have
$$f\circ \gamma(\lambda_0)>k>k_0>(1-\lambda)f(x_0)+\lambda f(y_0).$$
Since $f$ is continuous, we can assume that there exists $\mu>0$ such that for every $y\in B(x_0, \mu)$ and $y\in B(y_0,\mu)$ we get
\begin{equation}\label{jet-1}
  k_0>f(y).
\end{equation}
Suppose that $\mu^\prime>\mu$ is small enough such that $(T_{\mu^{\prime}},\phi_{\mu^\prime})$ is a fermi coordinate system around $\gamma(t)$. Define $\fii:T_{\mu^\prime}\rightarrow \RR$ by
$$\fii(y)= -\delta d^2(y, \gamma(\lambda_0))+C d_S^2(y)+k$$
where $\delta$ is such that
\begin{equation}\label{jet-2}
 k_0= -\delta(\max\{\lambda_0, 1-\lambda_0\}\norm \dot{\gamma}(0)\norm+\mu)^2+k,
\end{equation}
and $C$ is large enough such that
\begin{equation}\label{jet-3}
  C \mu^2+k_0\geq \sup_{\overline{T_{\mu^\prime}}}f,
\end{equation}
and $S:=\gamma[0,1]$. Note that $\fii$ is $C^2$ on $T_{\mu^{\prime}}$ and $f-\fii$ is a continuous function on $\overline{T_{\mu}}$. Thus $f-\fii$ admits its maximum on $\overline{T_{\mu}}$. In the following  we show that this is a local maximum.

If $y=(0, \xi)$ in fermi coordinate, then by using the fact that $\proj y=x_0$ we have
  \begin{align*}
   d^2(y, \gamma(\lambda_0))&\leq [ d(y, \proj y)+d(\proj y, \gamma(\lambda_0)]^2\\
   &\leq[\mu+\norm \dot{\gamma}(0)\norm \lambda_0]^2\\
   &\leq (\mu+ \max\{\lambda_0, 1-\lambda_0\}\norm \dot{\gamma}(0)\norm)^2
   \end{align*}
   Thus we have from (\ref{jet-2}) and (\ref{jet-1})
   $$\fii(y)\geq -\delta(\mu+ \max\{\lambda_0, 1-\lambda_0\}\norm \dot{\gamma}(0)\norm)^2+k=k_0>f(y).$$

If $y=(s,\xi)$ with $\norm \xi\norm=\mu$ and  $s\in (0,1)$, we have by (\ref{jet-3})
$$\fii(y)\geq -\delta[\mu+ \max\{\lambda_0, 1-\lambda_0\}\norm \dot{\gamma}(0)\norm]^2+C\mu^2+k =k_0+C\mu^2\geq \sup_{\overline{T_{\mu^\prime}}} f\geq f(y).$$
If $y=(1,\xi)$, we conclude similarly that $\fii(y)\geq f(y)$. Thus for every $y\in \bd T_{\mu}$ we get $\fii(y)\geq f(y)$. Note that $\fii(\gamma(\lambda_0))=k<f(\gamma(\lambda_0))$ which implies that $f-\fii$ has a maximum in the interior of $T_{\mu}$. Let $y_0$ be the point in $T_{\mu}$ such that $f-\fii$ gets its maximum at this point. By the definition of second order superjets we have $(d\fii(y_0),d^{2}\fii(y_0))\in J^{2,+}f(y_0)$. Note that by Theorem \ref{concave} there exists $X\neq 0$ in  the spaces of symmetric bilinear forms at $y_0$ such that $d^2d_S^2(y_0)(X,X)<0$. Since $-\delta d^2(y, \gamma(\lambda_0))$ is a concave function near $\gamma(\lambda_0)$, we get $d^2\fii(y_0)(X,X)<0$ which is a contradiction.
\end{proof}


\end{document}